\documentclass[12pt,a4paper]{article}
\usepackage{amssymb}

\usepackage{amsthm}
\usepackage[pdftex]{graphicx}
\usepackage{subfigure}

\newtheorem{definition}{Definition}[section]
\newtheorem{proposition}{Proposition}[section]
\newtheorem{theorem}{Theorem}[section]
\newtheorem{remark}{Remark}[section]
\newtheorem{corollary}{Corollary}[section]
\newtheorem{lemma}{Lemma}[section]

\begin{document}

\author{Nicoleta Aldea and Gheorghe Munteanu}
\title{On complex Landsberg and Berwald spaces}
\date{}
\maketitle

\begin{abstract}
In this paper, we study the complex Landsberg spaces and some of their
important subclasses. The tools of this study are: the Chern-Finsler,
Berwald and Rund complex linear connections. We introduce and characterize
the class of generalized Berwald and complex Landsberg spaces. The
intersection of these spaces gives the so called $G-$ Landsberg class. This
last class contains two other kinds of complex Finsler spaces: strong
Landsberg and $G-$K\"{a}hler spaces. We prove that the class of $G-$%
K\"{a}hler spaces coincides with complex Berwald spaces, in \cite{Ai}'s
sense, and it is a subclass of the strong Landsberg spaces. Some special
complex Finsler spaces with $(\alpha ,\beta )$ - metrics offer examples of
generalized Berwald spaces. Complex Randers spaces with generalized Berwald
and weakly K\"{a}hler properties are complex Berwald spaces.
\end{abstract}

\begin{flushleft}
\strut \textbf{2000 Mathematics Subject Classification:} 53B40, 53C60.

\end{flushleft}

\section{Introduction}

\setlength\arraycolsep{3pt} \setcounter{equation}{0}The real Landsberg
spaces, in particular the real Berwald spaces, have been a major subject of
study for many people over the years. In 1926 L. Berwald introduced a
special class of Finsler spaces which took his name in 1964. It is known
that a real Finsler space is called a Berwald space if the local
coefficients of the Berwald connection depend only on position coordinates.
An equivalent condition to this is that the Cartan tensor field is $h-$
parallel to the Berwald connection, i.e. $C_{ijk;r}=0,$ where here '$;$'
means the horizontal covariant derivative with respect to Berwald
connection. In 1934 Cartan emphasized two weak points of the Berwald
connection. One is that it is not metrical. Moreover $g_{ij;k}=-2C_{ijk;0}$
and therefore if $C_{ijk;0}=0$, then it becomes metrical. However, such a
space was called Landsberg by L. Berwald in 1928.

Many great contributions to the geometry of the real Landsberg and Berwald
spaces are due to Z. Szabo \cite{Sz1}, M. Matsumoto \cite{Mi}, P. Antonelli
\cite{An}, A. Bejancu \cite{Be}, Z. Shen \cite{Sh1}. Every Berwald space is
a Landsberg space. The converse, has been a long-standing problem, \cite
{L-S,Sz2,Do}.

Part of the general themes from real Finsler geometry about Landsberg and
Berwald spaces can be broached in complex Finsler geometry. However, there
are sensitive differences comparing to real reasonings, mainly on account of
the fact that in complex Finsler geometry there exist two different
covariant derivative for the Cartan tensors, $C_{i\bar{j}k\;|\;h}$ and $C_{i%
\bar{j}k\;|\;\bar{h}}$. Such reason determined T. Aikou, \cite{Ai}, to
request in the definition of a complex Berwald space, beside the natural
condition $C_{i\bar{j}k\;|\;h}=0$, the K\"{a}hler condition.

Therefore, the same arguments will be taken into account in the definition
of a complex Landsberg space. Using some ideas from the real case, related
to the Rund and Berwald connections, our aim in the present paper is to
introduce and study the complex Landsberg spaces and some of their
subclasses.

We associate to the canonical nonlinear connection, with the local
coefficients $\stackrel{c}{N_{j}^{i}}:=\dot{\partial}_{j}G^{i}$, two complex
linear connections: one of Berwald type $B\Gamma :=(\stackrel{c}{N_{j}^{i}},%
\stackrel{B}{L_{jk}^{i}},\stackrel{B}{L_{j\bar{k}}^{i}},0,0)$ and another of
Rund type $R\Gamma :=(\stackrel{c}{N_{j}^{i}},\stackrel{c}{L_{jk}^{i}},%
\stackrel{c}{L_{j\bar{k}}^{i}},0,0).$ In the real case, a Finsler space is
Landsberg if the Berwald and Rund connections coincide. But, in complex
Finsler geometry the things are considerably more difficult because in
general the $B\Gamma $ and $R\Gamma $ connections are not of $(1,0)$ - type.
Moreover, in the complex case alongside of the horizontal covariant
derivative, with respect to $B\Gamma $ connection, we have its conjugate.
Here, we speak of complex \textit{Landsberg} space iff $\stackrel{B}{%
L_{jk}^{i}}=\stackrel{c}{L_{jk}^{i}}$ and various characterizations of
Landsberg spaces are proved in Theorem 3.1. Further on we define the class
of $G$\textit{\ - Landsberg} spaces. It is in the class of complex Landsberg
spaces with $\dot{\partial}_{\bar{k}}G^{i}=0$. Theorem 3.2 reports on the
necessary and sufficient conditions for a complex Finsler space to be a $G$
- Landsberg space. A reinforcement of the tensorial characterization for a $%
G $ - Landsberg space gives rise to a subclass of $G$ - Landsberg, namely
\textit{strong Landsberg} iff $C_{l\bar{r}h\stackrel{B}{|}0}=0$ and $C_{j%
\bar{r}h\stackrel{B}{|}\bar{0}}=0.$ Other important properties of the strong
Landsberg spaces are contained in Theorem 3.3.

Because any K\"{a}hler space is a complex Landsberg space, the substitution
of the Landsberg condition with the K\"{a}hler condition in the definition
of the $G$ - Landsberg spaces leads to another subclass of this, called us $%
G $\textit{\ - K\"{a}hler.} Inter alia in Theorem 3.4 we prove that it
coincides with the category of complex Berwald spaces defined by Aikou in
\cite{Ai}. The strong Landsberg spaces are situated somewhere between
complex Berwald spaces and $G$ - Landsberg spaces.

The complex Berwald spaces were introduced as a generalization of the real
case, but in the particular context of K\"{a}hler. Therefore an
unquestionable extension of these, directly related on the $B\Gamma $
connection, is called by us \textit{generalized Berwald} space. It is with
the coefficients $\stackrel{B}{L_{jk}^{i}}$ depending only on the position $%
z.$ We give some characterizations for the generalized Berwald space, see
Theorem 3.6.

An intuitive scheme with the introduced classes of complex Finsler spaces is
in the following figure.
\begin{figure*}[h]
\centering
\includegraphics[scale=0.50]{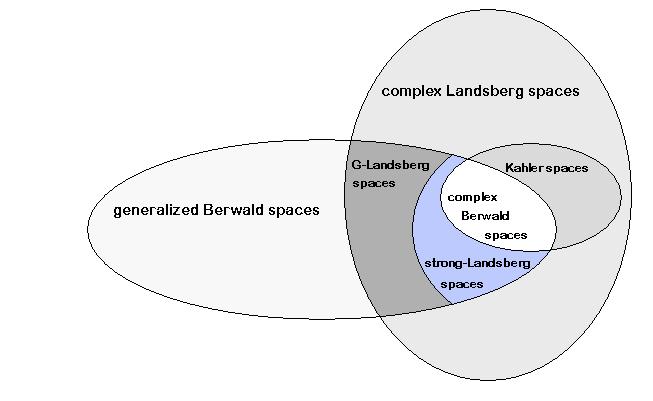} 
\caption{Inclusions}
\end{figure*}

The general theory on generalized Berwald spaces is completed by some
special outcomes for the class of complex Finsler spaces with $(\alpha
,\beta )$ - metrics. We prove that the complex Berwald spaces under
assumptions of generalized Berwald and weakly K\"{a}hler are complex
Berwald, see Theorem 4.3. A class of complex Kropina spaces which are
generalized Berwald is distinguished in Theorem 4.4.

The organization of the paper is as follows. In \S 2, we recall some
preliminary properties of the $n$ - dimensional complex Finsler spaces,
completed with some others needed for our aforementioned study. In \S 3, we
prove the above mentioned Theorems and we estabilish interrelations among
all classes of complex Finsler spaces. In section \S 4 we produced some
family of complex Finsler spaces with $(\alpha ,\beta )$ - metrics which are
generalized Berwald spaces and in particular, complex Berwald spaces.

\section{Preliminaries}

In this section we will give some preliminaries about complex Finsler
geometry with Chern-Finsler, Berwald and Rund complex linear connections. We
will set the basic notions (for more see \cite{A-P,Mub}) and we will prove
some important properties of these connections.

\subsection{Complex Finsler spaces}

\setcounter{equation}{0}Let $M$ be a $n-$dimensional complex manifold, $%
z=(z^k)_{k=\overline{1,n}}$ be the complex coordinates in a local chart.

The complexified of the real tangent bundle $T_CM$ splits into the sum of
holomorphic tangent bundle $T^{\prime }M$ and its conjugate $T^{\prime
\prime }M$. The bundle $T^{\prime }M$ is itself a complex manifold and the
local coordinates in a local chart will be denoted by $u=(z^k,\eta ^k)_{k=%
\overline{1,n}}.$ These are changed into $(z^{\prime k},\eta ^{\prime k})_{k=%
\overline{1,n}}$ by the rules $z^{\prime k}=z^{\prime k}(z)$ and $\eta
^{\prime k}=\frac{\partial z^{\prime k}}{\partial z^l}\eta ^l.$

A \textit{complex Finsler space} is a pair $(M,F)$, where $F:T^{\prime
}M\rightarrow \mathbb{R}^{+}$ is a continuous function satisfying the
conditions:

\textit{i)} $L:=F^2$ is smooth on $\widetilde{T^{\prime }M}:=T^{\prime
}M\backslash \{0\};$

\textit{ii)} $F(z,\eta )\geq 0$, the equality holds if and only if $\eta =0;$

\textit{iii)} $F(z,\lambda \eta )=|\lambda |F(z,\eta )$ for $\forall \lambda
\in \mathbb{C}$;

\textit{iv)} the Hermitian matrix $\left( g_{i\bar{j}}(z,\eta )\right) $ is
positive definite, where $g_{i\bar{j}}:=\frac{\partial ^2L}{\partial \eta
^i\partial \bar{\eta}^j}$ is the fundamental metric tensor. Equivalently, it
means that the indicatrix is strongly pseudo-convex.

Consequently, from $iii$) we have $\frac{\partial L}{\partial \eta ^k}\eta
^k=\frac{\partial L}{\partial \bar{\eta}^k}\bar{\eta}^k=L,$ $\frac{\partial
g_{i\bar{j}}}{\partial \eta ^k}\eta ^k=\frac{\partial g_{i\bar{j}}}{\partial
\bar{\eta}^k}\bar{\eta}^k=0$ and $L=g_{i\bar{j}}\eta ^i\bar{\eta}^j.$

Roughly speaking, the geometry of a complex Finsler space consists of the
study of the geometric objects of the complex manifold $T^{\prime }M$
endowed with the Hermitian metric structure defined by $g_{i\bar{j}}.$

Therefore, the first step is to study sections of the complexified tangent
bundle of $T^{\prime }M,$ which is decomposed in the sum $T_C(T^{\prime
}M)=T^{\prime }(T^{\prime }M)\oplus T^{\prime \prime }(T^{\prime }M)$. Let $%
VT^{\prime }M\subset T^{\prime }(T^{\prime }M)$ be the vertical bundle,
locally spanned by $\{\frac \partial {\partial \eta ^k}\}$, and $VT^{\prime
\prime }M$ be its conjugate.

At this point, the idea of complex nonlinear connection, briefly $(c.n.c.),$
is an instrument in 'linearization' of this geometry. A $(c.n.c.)$ is a
supplementary complex subbundle to $VT^{\prime }M$ in $T^{\prime }(T^{\prime
}M)$, i.e. $T^{\prime }(T^{\prime }M)=HT^{\prime }M\oplus VT^{\prime }M.$
The horizontal distribution $H_uT^{\prime }M$ is locally spanned by $\{\frac
\delta {\delta z^k}=\frac \partial {\partial z^k}-N_k^j\frac \partial
{\partial \eta ^j}\}$, where $N_k^j(z,\eta )$ are the coefficients of the $%
(c.n.c.)$. The pair $\{\delta _k:=\frac \delta {\delta z^k},\dot{\partial}%
_k:=\frac \partial {\partial \eta ^k}\}$ will be called the adapted frame of
the $(c.n.c.)$ which obey to the change rules $\delta _k=\frac{\partial
z^{\prime j}}{\partial z^k}\delta _j^{\prime }$ and $\dot{\partial}_k=\frac{%
\partial z^{\prime j}}{\partial z^k}\dot{\partial}_j^{\prime }.$ By
conjugation everywhere we have obtained an adapted frame $\{\delta _{\bar{k}%
},\dot{\partial}_{\bar{k}}\}$ on $T_u^{\prime \prime }(T^{\prime }M).$ The
dual adapted bases are $\{dz^k,\delta \eta ^k\}$ and $\{d\bar{z}^k,\delta
\bar{\eta}^k\}.$

Certainly, a main problem in this geometry is to determine a $(c.n.c.)$
related only to the fundamental function of the complex Finsler space $%
(M,F). $

The next step is the action of a derivative law $D$ on the sections of $%
T_C(T^{\prime }M).$ A Hermitian connection $D$, of $(1,0)-$ type, which
satisfies in addition $D_{JX}Y=JD_XY,$ for all $X$ horizontal vectors and $J$
the natural complex structure of the manifold, is so called Chern-Finsler
connection (cf. \cite{A-P}), in brief $C-F.$ The $C-F$ connection is locally
given by the following coefficients (cf. \cite{Mub}):
\begin{equation}
N_j^i=g^{\overline{m}i}\frac{\partial g_{l\overline{m}}}{\partial z^j}\eta
^l=L_{lj}^k\eta ^l\;;\;L_{jk}^i=g^{\overline{l}i}\delta _kg_{j\overline{l}}=%
\dot{\partial}_jN^i_k\;\;;\;C_{jk}^i=g^{\overline{l}i}\dot{\partial}_kg_{j%
\overline{l}},  \label{1.3}
\end{equation}
and $L_{\overline{j}k}^{\overline{\imath }}=C_{\overline{j}k}^{\overline{%
\imath }}=0$, where here and further on $\delta _k$ is the adapted frame of
the $C-F$ $(c.n.c.)$ and $D_{\delta _k}\delta _j=L_{jk}^i\delta _i,$ $D_{%
\dot{\partial}_k}\dot{\partial}_j=C_{jk}^i\dot{\partial}_i,$ etc. The $C-F$
connection is the main tool in this study.

Denoting by $"\shortmid "$ , $"\mid "$ , $"\bar{\shortmid}"$ and $"\bar{\mid}%
",$ the $h-$, $v-$, $\overline{h}-$, $\overline{v}-$ covariant derivatives
with respect to $C-F$ connection, respectively, it results
\begin{eqnarray}
\eta _{|k}^i &=&\eta _{|\overline{k}}^i=\eta ^i|_{\overline{k}}=0\;;\;\;\eta
^i|_k=\delta _k^i;  \label{1.3.} \\
g_{i\overline{j}|k} &=&g_{i\overline{j}|\overline{k}}=g_{i\overline{j}%
}|_k=g_{i\overline{j}}|_{\overline{k}}=0.  \nonumber
\end{eqnarray}

Now, we consider the complex Cartan tensors: $C_{i\bar{j}k}=\dot{\partial}_k
g_{i{\bar j}}$ and $C_{i\bar{j}\bar{k}}=\dot{\partial}_{\bar{k}} g_{i{\bar j}%
}$.

\begin{lemma}
For any complex Finsler space $(M,F)$, we have

i) $C_{l\bar{r}h|k}=(\dot{\partial}_{h}L_{lk}^{i})g_{i\bar{r}};$

ii) $C_{l\bar{r}\bar{h}|k}=(\dot{\partial}_{\bar{h}}L_{lk}^{i})g_{i\bar{r}}+(%
\dot{\partial}_{\bar{h}}N_{k}^{i})C_{i\bar{r}l}.$
\end{lemma}

\begin{proof}Differentiating $N_k^ig_{i\bar{r}}=\frac{\partial g_{j\bar{r}}}{%
\partial z^k}\eta ^j$ with respect to $\eta ^l$, gives
\begin{equation}
L_{lk}^ig_{i\bar{r}}=\frac{\partial g_{l\bar{r}}}{\partial z^k}-N_k^iC_{i%
\bar{r}l}.  \label{1.14}
\end{equation}

Now, differentiating in (\ref{1.14}) with respect to $\eta ^h$ it
results i), and with respect to $%
\bar{\eta}^h$ leads to ii).
\end{proof}

Let us recall that in \cite{A-P}'s terminology, the complex Finsler space $%
(M,F)$ is \textit{strongly K\"{a}hler} iff $T_{jk}^i=0,$ \textit{K\"{a}hler}$%
\;$iff $T_{jk}^i\eta ^j=0$ and \textit{weakly K\"{a}hler }iff\textit{\ } $%
g_{i\overline{l}}T_{jk}^i\eta ^j\overline{\eta }^l=0,$ where $%
T_{jk}^i:=L_{jk}^i-L_{kj}^i.$ In \cite{C-S} it is proved that strongly
K\"{a}hler and K\"{a}hler notions actually coincide. We notice that in the
particular case of complex Finsler metrics which come from Hermitian metrics
on $M,$ so-called \textit{purely Hermitian metrics} in \cite{Mub}, (i.e. $%
g_{i\overline{j}}=g_{i\overline{j}}(z)$)$,$ all those nuances of K\"{a}hler
are same.

On the other hand, as in Aikou's work \cite{Ai}, a complex Finsler space
which is K\"{a}hler and $L_{jk}^{i}=L_{jk}^{i}(z)$ is named a \textit{%
complex Berwald} space.

\subsection{Connections on complex Finsler spaces}

It is well known by \cite{A-P,Mub} that the complex geodesics curves are
defined by means of Chern-Finsler $(c.n.c.)$. Between complex spray and $%
(c.n.c.)$ there exists an interdependence, one determining the other. In
\cite{Mub} it is proved that the Chern-Finsler $(c.n.c.)$ does not generally
come from a complex spray except when the complex metric is weakly
K\"{a}hler. On the other hand, its local coefficients $N_{j}^{k}=g^{\bar{m}k}%
\frac{\partial g_{l\bar{m}}}{\partial z^{j}}\eta ^{l}$ always determine a
complex spray with coefficients $G^{i}=\frac{1}{2}N_{j}^{i}\eta ^{j}.$
Further, $G^{i}$ induce a $(c.n.c.)$ denoted by $\stackrel{c}{N_{j}^{i}}:=%
\dot{\partial}_{j}G^{i}$ and called \textit{canonical} in \cite{Mub}, where
it is proved that it coincides with Chern-Finsler $(c.n.c.)$ if and only if
the complex Finsler metric is K\"{a}hler. Using canonical $(c.n.c.)$ we
associate to it the next complex linear connections: one of Berwald type
\[
B\Gamma :=\left( \stackrel{c}{N_{j}^{i}},\stackrel{B}{L_{jk}^{i}}:=\dot{%
\partial}_{k}\stackrel{c}{N_{j}^{i}}=\stackrel{B}{L_{kj}^{i}},\stackrel{B}{%
L_{j\bar{k}}^{i}}:=\dot{\partial}_{\bar{k}}\stackrel{c}{N_{j}^{i}}%
,0,0\right)
\]
and another of Rund type
\[
R\Gamma :=\left( \stackrel{c}{N_{j}^{i}},\stackrel{c}{L_{jk}^{i}}:=\frac{1}{2%
}g^{\overline{l}i}(\stackrel{c}{\delta _{k}}g_{j\overline{l}}+\stackrel{c}{%
\delta _{j}}g_{k\overline{l}}),\stackrel{c}{L_{j\bar{k}}^{i}}:=\frac{1}{2}g^{%
\overline{l}i}(\stackrel{c}{\delta _{\bar{k}}}g_{j\overline{l}}-\stackrel{c}{%
\delta _{\bar{l}}}g_{j\overline{k}}),0,0\right) ,
\]
where $\stackrel{c}{\delta _{k}}:=\frac{\partial }{\partial z^{k}}-\stackrel{%
c}{N_{k}^{j}}\dot{\partial}_{j}.$ $R\Gamma $ is only $h-$ metrical and $%
B\Gamma $ is neither $h-$ nor $v-$ metrical, (for more details see \cite{Mub}%
). Note that the spray coefficients perform $2G^{i}:=N_{j}^{i}\eta ^{j}=%
\stackrel{c}{N_{j}^{i}}\eta ^{j}=\stackrel{B}{L_{jk}^{i}}\eta ^{j}\eta ^{k}$
and $\stackrel{c}{\delta _{j}}=\delta _{j}-(\stackrel{c}{N_{j}^{k}}%
-N_{j}^{k})\dot{\partial}_{k}.$

Moreover, in the K\"{a}hler case $\stackrel{c}{\delta _j}=$ $\delta _j$, (%
\cite{Mub}, p. 68) and hence, $\stackrel{c}{\delta _j}g_{k\bar{h}}=\stackrel{%
c}{\delta _k}g_{j\bar{h}}$ which contracted with $g^{\bar{\imath}j}$ gives $%
g^{\bar{\imath}j}(\stackrel{c}{\delta _j}g_{k\bar{h}}-\stackrel{c}{\delta _k}%
g_{j\bar{h}})=0,$ that is $\stackrel{c}{L_{\bar{h}j}^{\bar{\imath}}}=0$. By
conjugation, it follows $\stackrel{c}{L_{h\bar{j}}^i}=0.$ Also, in the
K\"{a}hler case we have $L_{jk}^i=\stackrel{c}{L_{jk}^i}=\stackrel{B}{%
L_{jk}^i}.$

Further on, everywhere in this paper the Berwald and Rund connections will
be specified by a super-index, like above (e.g. $\stackrel{c}{\delta _k}$, $%
\stackrel{B}{L_{jk}^i}$, $\stackrel{c}{L_{jk}^i}$, $X_{\stackrel{B}{|}k}$,
etc.), while for the Chern-Finsler connection will keep the initial generic
notation without super-index (e.g. $\delta _k$, $L_{jk}^i$, $X_{|k}$, etc.).

\begin{lemma}
For any complex Finsler space $(M,F)$, the following hold:

i) $\stackrel{B}{L_{j\bar{k}}^{i}}\bar{\eta}^{k}=0;$

ii) $-C_{l\bar{r}h\stackrel{B}{|}0}=g_{l\bar{r}\stackrel{B}{|}h}+g_{h\bar{r}%
\stackrel{B}{|}l}+\stackrel{B}{L_{\bar{r}h}^{\bar{m}}}g_{l\bar{m}}+\stackrel{%
B}{L_{\bar{r}l}^{\bar{m}}}g_{h\bar{m}};$

iii) $2(\dot{\partial}_{\bar{h}}G^{i})g_{i\bar{r}}=C_{0\bar{r}\bar{h}%
\stackrel{B}{|}0}=C_{0\bar{r}\bar{h}|0};$

iv) $C_{i\bar{j}h\stackrel{B}{|}k}=\dot{\partial}_{h}(g_{i\bar{j}\stackrel{B%
}{|}k})+(\dot{\partial}_{h}\stackrel{B\;\;}{L_{ik}^{l})}g_{l\bar{j}}+(\dot{%
\partial}_{h}\stackrel{B\;\;}{L_{\bar{j}k}^{\bar{m}})}g_{i\bar{m}};$

v) $C_{i\bar{r}\bar{h}\stackrel{B}{|}k}=\dot{\partial}_{\bar{h}}(g_{i\bar{j}%
\stackrel{B}{|}k})+(\dot{\partial}_{\bar{h}}\stackrel{B\;\;}{L_{ik}^{l})}g_{l%
\bar{j}}+(\dot{\partial}_{\bar{h}}\stackrel{B\;\;}{L_{\bar{j}k}^{\bar{m}})}%
g_{i\bar{m}}+\stackrel{B\;\;}{L_{k\bar{h}}^{l}}C_{i\bar{j}l}-\stackrel{B\;\;%
}{L_{\bar{h}k}^{\bar{m}}}C_{i\bar{j}\bar{m}},$

where $\stackrel{B}{\shortmid }$ is $h-$covariant derivative with respect to
$B\Gamma $ connection.
\end{lemma}

\begin{proof}
i)\textbf{\ }$\stackrel{B}{L_{j\bar{k}}^i}\bar{\eta}^k=\frac{\partial
^2G^i}{\partial \eta ^j\partial \bar{\eta}^k}\bar{\eta}^k=\dot{\partial}_j[(%
\dot{\partial}_{\bar{k}}G^i)\bar{\eta}^k]=\dot{\partial}_j[\dot{\partial}_{%
\bar{k}}(\frac 12N_l^i\eta ^l)\bar{\eta}^k]$

$=\frac 12\dot{\partial}_j[(\dot{\partial}_{\bar{k}}N_l^i)\eta ^l\bar{\eta}%
^k]=\frac 12\dot{\partial}_j[\dot{\partial}_{\bar{k}}(g^{\bar{m}i}\frac{%
\partial g_{s\bar{m}}}{\partial z^l}\eta ^s)\bar{\eta}^k\eta ^l]=0.$

ii) $G^i=\frac 12N_j^i\eta ^j=\frac 12g^{\bar{m}i}\frac{\partial g_{j\bar{m}}%
}{\partial z^h}\eta ^j\eta ^k$ can be rewrite as follows
\begin{equation}
G^ig_{i\bar{r}}=\frac 12\frac{\partial g_{j\bar{r}}}{\partial z^k}\eta
^j\eta ^k.  \label{1.9}
\end{equation}
Differentiating (\ref{1.9}) with respect to $\eta ^l$ yields
\begin{equation}
\stackrel{c}{N_l^i}g_{i\bar{r}}=\frac 12\left( \frac{\partial g_{l\bar{r}}}{%
\partial z^k}+\frac{\partial g_{k\bar{r}}}{\partial z^l}\right) \eta
^k-G^iC_{i\bar{r}l}.  \label{1.10}
\end{equation}
Differentiating (\ref{1.10}) with respect to $\eta ^h$ it results
\begin{equation}
\stackrel{B}{L_{hl}^i}g_{i\bar{r}}=\frac 12\left( \frac{\partial g_{l\bar{r}}%
}{\partial z^h}+\frac{\partial g_{h\bar{r}}}{\partial z^l}\right) +\frac 12%
\frac{\partial C_{l\bar{r}h}}{\partial z^k}\eta ^k-G^i(\dot{\partial}_hC_{i%
\bar{r}l})-\stackrel{c}{N_h^i}C_{i\bar{r}l}-\stackrel{c}{N_l^i}C_{i\bar{r}h},
\label{1.11}
\end{equation}
which leads to
\begin{equation}
-(\stackrel{c}{\delta _k}C_{l\bar{r}h})\eta ^k+\stackrel{c}{N_h^i}C_{i\bar{r}%
l}+\stackrel{c}{N_l^i}C_{i\bar{r}h}=\stackrel{c}{\delta _h}g_{l\bar{r}}+%
\stackrel{c}{\delta _l}g_{h\bar{r}}-2\stackrel{B}{L_{hl}^i}g_{i\bar{r}}.
\label{1.12}
\end{equation}
Now, taking into account that $g_{l\bar{r}\stackrel{B}{|}h}=\stackrel{c}{%
\delta _h}g_{l\bar{r}}-\stackrel{B}{L_{lh}^i}g_{i\bar{r}}-\stackrel{B}{L_{%
\bar{r}h}^{\bar{m}}}g_{l\bar{m}}$ and i) it results ii).

iii) Differentiating (\ref{1.9}) with respect to $\bar{\eta}^h$ yields
\begin{equation}
2(\dot{\partial}_{\bar{h}}G^i)g_{i\bar{r}}=\frac{\partial C_{l\bar{r}\bar{h}}%
}{\partial z^k}\eta ^l\eta ^k-\stackrel{c}{N_k^i}\eta ^kC_{i\bar{r}\bar{h}}.
\label{1.13}
\end{equation}
But, $\dot{\partial}_l(C_{j\bar{s}\bar{h}}\eta ^j)=(\dot{\partial}_lC_{j\bar{%
s}\bar{h}})\eta ^j+C_{l\bar{s}\bar{h}}=\dot{\partial}_{\bar{h}}(C_{j\bar{s}%
l}\eta ^j)+C_{l\bar{s}\bar{h}}=C_{l\bar{s}\bar{h}}.$ Using (\ref{1.13}) and $%
N_j^i\eta ^j=\stackrel{c}{N_j^i}\eta ^j$ we obtain $2(\dot{\partial}_{\bar{h}%
}G^i)g_{i\bar{r}}=\stackrel{c}{\delta _k}(C_{l\bar{r}\bar{h}}\eta ^l)\eta
^k=\delta _k(C_{l\bar{r}\bar{h}}\eta ^l)\eta ^k$, which together with i) and the
$h$ - covariant derivative rule with respect to $C-F$ connection gives iii).

iv) Using again $g_{i\bar{j}\stackrel{B}{|}k}=\stackrel{c}{\delta _k}g_{i%
\bar{j}}-\stackrel{B}{L_{ik}^l}g_{l\bar{j}}-\stackrel{B}{L_{\bar{j}k}^{\bar{m%
}}}g_{i\bar{m}}$, which differentiated with respect to $\eta ^h$, gives

$\dot{\partial}_h(g_{i\bar{j}\stackrel{B}{|}k})=\frac{\partial C_{i\bar{j}h}%
}{\partial z^k}-\stackrel{B\;\;}{L_{hk}^l}C_{i\bar{j}l}-\stackrel{c}{N_k^l}(%
\dot{\partial}_hC_{i\bar{j}l})-(\dot{\partial}_h\stackrel{B\;\;}{L_{ik}^l)}%
g_{l\bar{j}}-\stackrel{B\;\;}{L_{ik}^l}C_{l\bar{j}h}$

$-(\dot{\partial}_h\stackrel{B\;\;}{L_{\bar{j}k}^{\bar{m}})}g_{i\bar{m}}-%
\stackrel{B\;\;}{L_{\bar{j}k}^{\bar{p}}}C_{i\bar{m}h}$

$=\stackrel{c}{\delta _k}C_{i\bar{j}h}-\stackrel{B\;\;}{L_{hk}^l}C_{i\bar{j}%
l}-\stackrel{B\;\;}{L_{ik}^l}C_{l\bar{j}h}-\stackrel{B\;\;}{L_{\bar{j}k}^{%
\bar{m}}}C_{i\bar{m}h}-(\dot{\partial}_h\stackrel{B\;\;}{L_{ik}^l)}g_{l\bar{j%
}}-(\dot{\partial}_h\stackrel{B\;\;}{L_{\bar{j}k}^{\bar{m}})}g_{i\bar{m}}$

$=C_{i\bar{j}h\stackrel{B}{|}k}\;-(\dot{\partial}_h\stackrel{B\;\;}{L_{ik}^l)%
}g_{l\bar{j}}-(\dot{\partial}_h\stackrel{B\;\;}{L_{\bar{j}k}^{\bar{m}})}g_{i%
\bar{m}},$ that is iv).

For v) we compute

$\dot{\partial}_{\bar{h}}(g_{i\bar{j}\stackrel{B}{|}k})=\frac{\partial C_{i%
\bar{j}\bar{h}}}{\partial z^k}-\stackrel{B\;\;}{L_{k\bar{h}}^l}C_{i\bar{j}%
l}-N_k^l(\dot{\partial}_lC_{i\bar{j}\bar{h}})-(\dot{\partial}_{\bar{h}}%
\stackrel{B\;\;}{L_{ik}^l)}g_{l\bar{j}}-\stackrel{B\;\;}{L_{ik}^l}C_{l\bar{j}%
\bar{h}}$

$-(\dot{\partial}_{\bar{h}}\stackrel{B\;\;}{L_{\bar{j}k}^{\bar{m}})}g_{i\bar{%
m}}-\stackrel{B\;\;}{L_{\bar{j}k}^{\bar{m}}}C_{i\bar{p}\bar{m}}$

$=\delta _kC_{i\bar{j}\bar{h}}-\stackrel{B\;\;}{L_{ik}^l}C_{l\bar{j}\bar{h}}-%
\stackrel{B\;\;}{L_{\bar{h}k}^{\bar{m}}}C_{i\bar{j}\bar{m}}-\stackrel{B\;\;}{%
L_{\bar{j}k}^{\bar{m}}}C_{i\bar{m}\bar{h}}-\stackrel{B\;\;}{L_{k\bar{h}}^l}%
C_{i\bar{j}l}+\stackrel{B\;\;}{L_{\bar{h}k}^{\bar{m}}}C_{i\bar{j}\bar{m}}$

$-(\dot{\partial}_{\bar{h}}\stackrel{B\;\;}{L_{ik}^l)}g_{l\bar{j}}-(\dot{%
\partial}_{\bar{h}}\stackrel{B\;\;}{L_{\bar{j}k}^{\bar{m}})}g_{i\bar{m}}$

$=C_{i\bar{j}\bar{h}\stackrel{B}{|}k}-\stackrel{B\;\;}{L_{k\bar{h}}^l}C_{i%
\bar{j}l}+\stackrel{B\;\;}{L_{\bar{h}k}^{\bar{m}}}C_{i\bar{j}\bar{m}}-(\dot{%
\partial}_{\bar{h}}\stackrel{B\;\;}{L_{ik}^l)}g_{l\bar{j}}-(\dot{\partial}_{%
\bar{h}}\stackrel{B\;\;}{L_{\bar{j}k}^{\bar{m}})}g_{i\bar{m}}.$
\end{proof}

\section{The complex Landsberg spaces}

\setcounter{equation}{0}In real Finsler geometry the classes of Landsberg
and Berwald spaces are related by Rund and Berwald connections, \cite{Be}.
Namely, a real Finsler space is Landsberg if the Berwald and Rund
connections coincide. Nevertheless, in complex Finsler geometry some
differences appear. We will speak about of three kinds of complex spaces of
Landsberg type.

\begin{definition}
Let $(M,F)$ be a $n$ - dimensional complex Finsler space. $(M,F)$ is called
complex Landsberg space if $\stackrel{B}{L_{jk}^{i}}=\stackrel{c}{L_{jk}^{i}}
$.
\end{definition}

We remark that any complex Finsler space which is K\"{a}hler is a Landsberg
space too, because if $(M,F)$ is K\"{a}hler, then $L_{jk}^{i}=\stackrel{c}{%
L_{jk}^{i}}=\stackrel{B}{L_{jk}^{i}}$ and so it is Landsberg. Thus, the
K\"{a}hler spaces offer a asset family of complex Landsberg spaces.

\begin{theorem}
Let $(M,F)$ be a $n$ - dimensional complex Finsler space. Then the following
assertions are equivalent:

i) $(M,F)$ is a complex Landsberg space;

ii) $C_{l\bar{r}h\stackrel{B}{|}0}=0;$

iii) $2(\dot{\partial}_{h}\stackrel{B}{L_{jk}^{i}})g_{i\bar{r}}-\stackrel{B}{%
L_{\bar{r}k}^{\bar{m}}}C_{j\bar{m}h}-\stackrel{B}{L_{\bar{r}j}^{\bar{m}}}C_{k%
\bar{m}h}=C_{j\bar{r}h\stackrel{B}{|}k}+C_{k\bar{r}h\stackrel{B}{|}j};$

iv) $g_{i\bar{j}\stackrel{B}{|}k}=\stackrel{c}{(L_{\bar{j}k}^{\bar{m}}}-%
\stackrel{B}{L_{\bar{j}k}^{\bar{m}}})g_{i\bar{m}}.$
\end{theorem}

\begin{proof}
i) $\Leftrightarrow $ ii). A direct computation gives
\[
g_{l\bar{r}\stackrel{B}{|}h}+g_{h\bar{r}\stackrel{B}{|}l}+\stackrel{B}{L_{%
\bar{r}h}^{\bar{m}}}g_{l\bar{m}}+\stackrel{B}{L_{\bar{r}l}^{\bar{m}}}g_{h%
\bar{m}}=2(\stackrel{c}{L_{lh}^i}g_{i\bar{r}}-\stackrel{B}{L_{lh}^i}g_{i\bar{%
r}})
\]
which, together with Lemma 2.2 ii), get the proof.

i) $\Rightarrow $ iii). Because $(M,F)$ is Landsberg we have
\[
\stackrel{B}{L_{jk}^i}g_{i\bar{r}}=\frac 12(\stackrel{c}{\delta _k}g_{j%
\overline{r}}+\stackrel{c}{\delta _j}g_{k\overline{r}}).
\]
Differentiating it with respect to $\eta ^h$ yields iii).

iii) $\Rightarrow $ ii). Contracting iii) with $\eta ^k$ gives $2(\dot{%
\partial}_h\stackrel{B}{L_{jk}^i})g_{i\bar{r}}\eta ^k=C_{j\bar{r}h\stackrel{B%
}{|}0}.$ On the other hand, $(\dot{\partial}_h\stackrel{B}{L_{jk}^i})g_{i%
\bar{r}}\eta ^k=0.$ From here it results ii).

i) $\Leftrightarrow $ iv). $g_{i\bar{j}\stackrel{B}{|}k}=\stackrel{c}{\delta
_k}g_{i\bar{j}}-\stackrel{B}{L_{ik}^l}g_{l\bar{j}}-\stackrel{B}{L_{\bar{j}%
k}^{\bar{m}}}g_{i\bar{m}}$

$=g_{i\bar{j}\stackrel{R}{|}k}+(\stackrel{c}{L_{ik}^l}-\stackrel{B}{L_{ik}^l}%
)g_{l\bar{j}}+(\stackrel{c}{L_{\bar{j}k}^{\bar{m}}}-\stackrel{B}{L_{\bar{j}%
k}^{\bar{m}}})g_{i\bar{m}},$ where $\stackrel{R}{\shortmid }$ is $h-$
covariant derivative with respect to $R\Gamma $ connection. But, $R\Gamma $
connection is $h-$ metrical, therefore $g_{i\bar{j}\stackrel{B}{|}k}=(%
\stackrel{c}{L_{ik}^l}-\stackrel{B}{L_{ik}^l})g_{l\bar{j}}-(\stackrel{c}{L_{%
\bar{j}k}^{\bar{m}}}-\stackrel{B}{L_{\bar{j}k}^{\bar{m}}})g_{i\bar{m}},$
which justifies this equivalence.
\end{proof}

\begin{definition}
Let $(M,F)$ be a $n$ - dimensional complex Finsler space. $(M,F)$ is called $%
G$ - Landsberg space if it is Landsberg and the spray coefficients $G^{i}$
are holomorphic in $\eta $, i.e. $\dot{\partial}_{\bar{k}}G^{i}=0$.
\end{definition}

Some immediately consequences follow below.

\begin{proposition}
If $(M,F)$ is a $G$ - Landsberg space then the connection $B\Gamma $ is of $%
(1,0)$ - type.
\end{proposition}

\begin{corollary}
$G^{i}$ are holomorphic in $\eta $ if and only if the connection $B\Gamma $
is of $(1,0)$ - type.
\end{corollary}

\begin{proposition}
$G^{i}$ are holomorphic in $\eta $ if and only if $\stackrel{B}{L_{jk}^{i}}$
depend only on $z$.
\end{proposition}

\begin{proof}
If $G^i$ are holomorphic functions in $\eta $, then $\dot{%
\partial}_{\bar{k}}G^i=0$ which leads to $\dot{\partial}_{\bar{k}}\stackrel{c%
}{N_h^i}=0$ and more $\dot{\partial}_{\bar{k}}\stackrel{B}{L_{jh}^i}=0$.
Hence, the functions $\stackrel{B}{L_{jh}^i}$ are holomorphic in $\eta ,$ too.

Now, we make a similar reasoning like that in \cite{C-S2}, Proposition 1.1,
but for the $0-$homogeneous functions $\stackrel{B}{L_{jh}^i}$ in $\eta$.
We consider $D_\varepsilon :=\{\eta \in T_z^{\prime }M$ $|$ $\mathcal{G}%
(\eta ,\bar{\eta})<\varepsilon ^2,$ $\varepsilon >0\}$ and we study the
functions $\stackrel{B}{L_{jh}^i}$ on the domain $D_\varepsilon $ $%
\backslash $ $D_{\frac 1\varepsilon }.$ Because these functions are $0$ -
homogeneous, their modulus achieve maximum at an interior point of $%
D_\varepsilon $ $\backslash $ $D_{\frac 1\varepsilon }.$ Thus, we can apply
the strong maximum principle which gives that the functions $\stackrel{B}{%
L_{jh}^i}$ are constant with respect to $\eta $ on $D_\varepsilon $ $%
\backslash $ $D_{\frac 1\varepsilon }.$ Now, making $\varepsilon \rightarrow
$ $\infty $ it results the functions $\stackrel{B}{L_{jh}^i}$ depend on $z$
only on $T_z^{\prime }M$ $\backslash $ $\{0\}.$ Under a change of the local
coordinates $(z^i,\eta ^i)$ into $(z^{\prime i},\eta ^{\prime i}),$ the
functions $\stackrel{B}{L_{jh}^i}$ are modified by the rule $\stackrel{B}{%
L_{jk}^{\prime i}}=\frac{\partial z^{\prime i}}{\partial z^r}\frac{\partial
z^s}{\partial z^{\prime j}}\frac{\partial z^q}{\partial z^{\prime k}}%
\stackrel{B}{L_{sq}^r}+\frac{\partial z^{\prime i}}{\partial z^r}\frac{%
\partial ^2z^r}{\partial z^{\prime j}\partial z^{\prime k}}.$ It results
that $\stackrel{B}{L_{jk}^{\prime i}}$ depend only on $z,$ too. Thus, globally we have $%
\stackrel{B}{L_{jk}^i}(z).$

Conversely, if $\stackrel{B}{L_{jk}^i}(z)$ then $\dot{\partial}_{\bar{k}}%
\stackrel{B}{L_{jh}^i}=0$, which contracted by $\eta ^j\eta ^h$ complete the
proof.
\end{proof}

\begin{corollary}
$\stackrel{B}{L_{jk}^{i}}$ depend only on $z$ if and only if $\dot{\partial}%
_{k}\stackrel{B}{L_{jh}^{i}}=0.$
\end{corollary}

\begin{proof}
It is obvious that\textbf{\ }$\stackrel{B}{L_{jk}^i}$ depend only
on $z$ implies $\dot{\partial}_k\stackrel{B}{L_{jh}^i}=0.$ Conversely, if $%
\dot{\partial}_k\stackrel{B}{L_{jh}^i}=0,$ by conjugation we have $\dot{%
\partial}_{\bar{k}}\stackrel{B}{L_{\bar{j}\bar{h}}^{\bar{\imath}}}=0,$ i.e. $%
\stackrel{B}{L_{\bar{j}\bar{h}}^{\bar{\imath}}}$ are holomorphic in $\eta .$
But, the functions $\stackrel{B}{L_{\bar{j}\bar{h}}^{\bar{\imath}}}$ are also
$0$ - homogeneous and so, by same arguments as in the above Proposition it
results that $\dot{\partial}_k\stackrel{B}{L_{\bar{j}\bar{h}}^{\bar{\imath}}}%
=0,$ and by conjugation $\dot{\partial}_{\bar{k}}\stackrel{B}{L_{jh}^i}=0.$
Applying again above Proposition we obtain $\stackrel{B}{L_{jk}^i}%
(z).$
\end{proof}

\begin{theorem}
Let $(M,F)$ be a $n$ - dimensional complex Finsler space. Then the following
assertions are equivalent:

i) $(M,F)$ is a $G$ - Landsberg space;

ii) $\stackrel{B}{L_{jk}^{i}}=\stackrel{c}{L_{jk}^{i}}(z);$

iii) $C_{l\bar{r}h\stackrel{B}{|}0}=0$ and $C_{j\bar{0}h\stackrel{B}{|}\bar{0%
}}=0;$

iv) $g_{i\bar{j}\stackrel{B}{|}k}=\stackrel{c}{L_{\bar{j}k}^{\bar{m}}}g_{i%
\bar{m}}$ and $\dot{\partial}_{\bar{h}}G^{i}=0.$

v) $C_{j\bar{r}h\stackrel{B}{|}k}+C_{k\bar{r}h\stackrel{B}{|}j}=0$ and $C_{r%
\bar{l}h\stackrel{B}{|}\bar{k}}+C_{r\bar{k}h\stackrel{B}{|}\bar{l}}=0.$
\end{theorem}

\begin{proof}
i) $\Leftrightarrow $ ii) is obtained by Proposition 3.1. i) $%
\Leftrightarrow $ iii) results by Lemma 2.2 iii) and Theorem 3.1 ii). Under
assumptions $\dot{\partial}_{\bar{h}}G^i=0,$ the equivalence i) $%
\Leftrightarrow $ iv) from Theorem 3.1 gets the proof for i) $%
\Leftrightarrow $ iv).

i) $\Rightarrow $ v) If $(M,F)$ is a $G$ - Landsberg space then by Lemma 2.2
ii) and v) we have $g_{l\bar{r}\stackrel{B}{|}h}+g_{h\bar{r}\stackrel{B}{|}%
l}=0$ and $C_{l\bar{r}\bar{h}\stackrel{B}{|}k}=\dot{\partial}_{\bar{h}}(g_{l%
\bar{r}\stackrel{B}{|}k}).$ So that

$C_{l\bar{r}\bar{h}\stackrel{B}{|}k}+C_{k\bar{r}\bar{h}\stackrel{B}{|}l}=%
\dot{\partial}_{\bar{h}}(g_{l\bar{r}\stackrel{B}{|}k}+g_{k\bar{r}\stackrel{B%
}{|}l})=0$ and by conjugation it leads to

$C_{r\bar{l}h\stackrel{B}{|}\bar{k}}+C_{r\bar{k}h\stackrel{B}{|}\bar{l}}=0.$
Now, using Lemma 2.2 iv) and Proposition 3.2 it results $C_{j\bar{r}h%
\stackrel{B}{|}k}+C_{k\bar{r}h\stackrel{B}{|}j}=0.$

v) $\Rightarrow $ i) First, contracting with $\eta ^k$ the identity $C_{j%
\bar{r}h\stackrel{B}{|}k}+C_{k\bar{r}h\stackrel{B}{|}j}=0$ it results $C_{j%
\bar{r}h\stackrel{B}{|}0}=0,$ i.e. the space is Landsberg. On the other hand
the contraction by $\bar{\eta}^k\bar{\eta}^l$ of the identity $C_{r\bar{l}h%
\stackrel{B}{|}\bar{k}}+C_{r\bar{k}h\stackrel{B}{|}\bar{l}}=0$ gives $2C_{r%
\bar{0}h\stackrel{B}{|}\bar{0}}=0$ and by conjugation it is $2C_{0\bar{r}%
\bar{h}\stackrel{B}{|}0}=0.$ But, by Lemma 2.2 iii) we have $2(\dot{\partial}%
_{\bar{h}}G^i)g_{i\bar{r}}=C_{0\bar{r}\bar{h}\stackrel{B}{|}0}.$ From here
we obtain $\dot{\partial}_{\bar{h}}G^i=0$, which completes the proof.
\end{proof}

Now, having in mind the tensorial characterization iii) from Theorem 3.2 for
a $G$ - Landsberg space, it give rise to another class of complex Landsberg
spaces.

\begin{definition}
Let $(M,F)$ be a $n$ - dimensional complex Finsler space. $(M,F)$ is called
a strong Landsberg space if $C_{l\bar{r}h\stackrel{B}{|}0}=0$ and $C_{j\bar{r%
}h\stackrel{B}{|}\bar{0}}=0.$
\end{definition}

\begin{theorem}
Let $(M,F)$ be a $n$ - dimensional complex Finsler space. Then the following
assertions are equivalent:

i) $(M,F)$ is a strong Landsberg space;

ii) $g_{l\bar{r}\stackrel{B}{|}s}(z)$ and $\dot{\partial}_{\bar{h}}G^{i}=0;$

iii) $C_{l\bar{r}h\stackrel{B}{|}k}=0$ and $\dot{\partial}_{\bar{h}}G^{i}=0;$

iv) $C_{j\bar{r}h\stackrel{B}{|}\bar{k}}=0.$
\end{theorem}

\begin{proof}
i) $\,\Rightarrow $ ii). If $(M,F)$ is a strong Landsberg space, then by
Theorem 3.2 iii) it is $G$ - Landsberg. Therefore, Lemma 2.2 iv) and v)
become

$C_{i\bar{j}h\stackrel{B}{|}k}=\dot{\partial}_h(g_{i\bar{j}\stackrel{B}{|}k})
$ and $C_{l\bar{r}\bar{h}\stackrel{B}{|}k}=\dot{\partial}_{\bar{h}}(g_{l\bar{%
r}\stackrel{B}{|}k}),$ which contracted by $\eta ^k$ implies
\begin{equation}
0=\dot{\partial}_h(g_{i\bar{j}\stackrel{B}{|}k})\eta ^k\;;\;0=\dot{\partial}%
_{\bar{h}}(g_{l\bar{r}\stackrel{B}{|}k})\eta ^k.  \label{2.1}
\end{equation}

Differentiating the second equation of (\ref{2.1}) with respect to $\eta ^s,$
it yields $0=\dot{\partial}_{\bar{h}}[\dot{\partial}_s(g_{l\bar{r}\stackrel{B%
}{|}k})\eta ^k]+\dot{\partial}_{\bar{h}}(g_{l\bar{r}\stackrel{B}{|}s}).$
Now, using the first relation from (\ref{2.1}) it results $\dot{\partial}_{%
\bar{h}}(g_{l\bar{r}\stackrel{B}{|}s})=0.$ Because $g_{l\bar{r}\stackrel{B}{|%
}s}$ are holomorphic with respect to $\eta $ and homogeneous of zero degree,
these give that $g_{l\bar{r}\stackrel{B}{|}s}$ depends on $z$ only, i.e. $%
g_{l\bar{r}\stackrel{B}{|}s}(z).$

Now, the conditions $g_{l\bar{r}\stackrel{B}{|}s}(z)$ and $\dot{\partial}_{%
\bar{h}}G^i=0$ substituted into Lemma 2.2 iv), gives $C_{i\bar{j}h\stackrel{B%
}{|}k}=0.$ So that, we have proved ii) $\Rightarrow $ iii).

To prove iii) $\Rightarrow $ iv) we use again Lemma 2.2 iv). Under
assumptions iii), it is $\dot{\partial}_h(g_{i\bar{j}\stackrel{B}{|}k})=0$
and by conjugation $\dot{\partial}_{\bar{h}}(g_{j\bar{\imath}\stackrel{B}{|}%
\bar{k}})=0.$ This means that $g_{j\bar{\imath}\stackrel{B}{|}\bar{k}}$ is
holomorphic in $\eta $ which together its $0$ - homogeneity implies $g_{j%
\bar{\imath}\stackrel{B}{|}\bar{k}}(z)$ and so its conjugate $g_{i\bar{j}%
\stackrel{B}{|}k}$ depends on $z$ only. Therefore, v) from Lemma 2.1 leads
to $C_{l\bar{r}\bar{h}\stackrel{B}{|}k}=0,$ i.e. iv).

The proof is complete if we show that iv) $\Rightarrow $ i). Indeed, $C_{l%
\bar{r}\bar{h}\stackrel{B}{|}k}=0$ implies $C_{l\bar{r}\bar{h}\stackrel{B}{|}%
0}=0$ and $\dot{\partial}_{\bar{h}}G^i=0,$ by Lemma 2.1 iii).

Lemma 2.2 v) gives $\dot{\partial}_{\bar{h}}(g_{i\bar{j}\stackrel{B}{|}k})=0$
and so $g_{j\bar{\imath}\stackrel{B}{|}\bar{k}}(z).$ Thus, by Lemma 2.2 iv)
we obtain $C_{l\bar{r}h\stackrel{B}{|}k}=0$, which contracted by $\eta ^k$ yilds $%
C_{l\bar{r}h\stackrel{B}{|}0}=0.$ So that the space is strong Landsberg.
\end{proof}

\begin{remark}
By Theorem 3.2 iv) and Theorem 3.3 ii) it results that a $n$ - dimensional
complex Finsler space is strong Landsberg if and only if $\stackrel{c}{(L_{%
\bar{j}k}^{\bar{m}}}g_{i\bar{m}})(z)$ and $\dot{\partial}_{\bar{h}}G^{i}=0.$
\end{remark}

Now recall that according to Aikou, \cite{Ai}, a complex Berwald space is a
Finsler space which is K\"{a}hler and $L_{jk}^i=L_{jk}^i(z)$.

Having in the mind that any K\"{a}hler complex Finsler space is Landsberg,
we can introduce another generalization for the $G$ - Landsberg spaces. So,
by replacing the Landsberg condition from definition of the $G$ - Landsberg
space with the K\"{a}hler condition we obtain:

\begin{definition}
Let $(M,F)$ be a $n$ - dimensional complex Finsler space. $(M,F)$ is called $%
G$ - K\"{a}hler space if it is K\"{a}hler and the spray coefficients $G^{i}$
are holomorphic in $\eta .$
\end{definition}

Some necessary and sufficient conditions for $G$ - K\"{a}hler spaces are
contained in the next theorem.

\begin{theorem}
Let $(M,F)$ be a $n$ - dimensional complex Finsler space. Then the following
assertions are equivalent:

i) $(M,F)$ is $G$ - K\"{a}hler;

ii) $\stackrel{B}{L_{j\bar{k}}^{i}}=\stackrel{c}{L_{j\bar{k}}^{i}}$;

iii) $\stackrel{B}{L_{jk}^{i}}=L_{jk}^{i}(z);$

iv) $(M,F)$ is a complex Berwald space;

v) $g_{i\bar{j}\;\stackrel{B}{|}\;k}=0$ and $\dot{\partial}_{\bar{h}%
}G^{i}=0. $
\end{theorem}

\begin{proof}
i) $\Leftrightarrow $ ii). If $(M,F)$ is $G$ - K\"{a}hler then $%
\stackrel{c}{L_{j\bar{k}}^i}=0$ and $\stackrel{B}{L_{j\bar{k}}^i}=0$ which
imply $\stackrel{B}{L_{j\bar{k}}^i}=\stackrel{c}{L_{j\bar{k}}^i}.$
Conversely, if $\stackrel{B}{L_{j\bar{k}}^i}=\stackrel{c}{L_{j\bar{k}}^i}$
then $\dot{\partial}_{\bar{k}}\stackrel{c}{N_j^i}=\frac 12g^{\overline{l}i}(%
\stackrel{c}{\delta _{\bar{k}}}g_{j\overline{l}}-\stackrel{c}{\delta _{\bar{l%
}}}g_{j\overline{k}})$ which contracted by $\bar{\eta}^k$ gives $\bar{\eta}%
^k(\stackrel{c}{\delta _{\bar{k}}}g_{j\overline{l}})-(\stackrel{c}{\delta _{%
\bar{l}}}g_{j\overline{k}})\bar{\eta}^k=0.$ But, $\bar{\eta}^k\stackrel{c}{%
\delta _{\bar{k}}}=\bar{\eta}^k\delta _{\bar{k}}$ and $(\stackrel{c}{\delta
_{\bar{l}}}g_{j\overline{k}})\bar{\eta}^k=({\delta _{\bar{l}}}g_{j\overline{k%
}})\bar{\eta}^k$. So that, we obtain $({\delta _{\bar{k}}}g_{j\overline{l}}-{%
\delta _{\bar{l}}}g_{j\overline{k}})\bar{\eta}^k=0.$ It implies that $(M,F)$
is K\"{a}hler, as well as $\dot{\partial}_{\bar{k}}\stackrel{c}{N_j^i}=0.$
The contraction of $\dot{\partial}_{\bar{k}}\stackrel{c}{N_j^i}=0$ by $\eta
^j$ gives $\dot{\partial}_{\bar{k}}G^i=0,$ i.e. $G^i$ does not depend on $\bar{%
\eta}^k.$

Taking into account $(M,F)$ is K\"{a}hler if and only if $L_{jk}^i=\stackrel{%
c}{L_{jk}^i}=\stackrel{B}{L_{jk}^i}$, and using Propositions 3.2, it follows the
proof for i) $\Leftrightarrow $ iii) and i) $\Leftrightarrow $ iv).

i) $\Leftrightarrow $ v). It is obvious that if $(M,F)$ is $G$ -
K\"{a}hler then $g_{i\bar{j}\stackrel{B}{|}k}=g_{i\bar{j}|k}=0.$ Conversely,
if $g_{i\bar{j}\stackrel{B}{|}k}=0$ and $\dot{\partial}_{\bar{h}}G^i=0$, then
$\stackrel{B}{L_{jk}^i}=g^{\overline{l}i}(\stackrel{c}{\delta _k}g_{j%
\overline{l}}).$ But, $\stackrel{B}{L_{jk}^i}=\stackrel{B}{L_{kj}^i}$ which
implies $\stackrel{c}{\delta _k}g_{j\overline{l}}=\stackrel{c}{\delta _j}g_{k%
\overline{l}}$ as well as $(\stackrel{c}{\delta _k}g_{j\overline{l}})\eta
^k=(\stackrel{c}{\delta _j}g_{k\overline{l}})\eta ^k$ i.e. K\"{a}hler.
\end{proof}

An immediately consequence of above Theorem follows.

\begin{proposition}
$(M,F)$ is a complex Berwald space if and only if the connections $B\Gamma $
and $R\Gamma $ are of $(1,0)$ - type.
\end{proposition}

\begin{lemma}
For any complex Finsler space $(M,F)$, $C_{l\bar{r}h|k}=0$ if and only if $%
C_{l\bar{r}h|\bar{k}}=0.$
\end{lemma}

\begin{proof}
If $C_{l\bar{r}h|k}=0$, then Lemma 2.1 i) induces $\dot{\partial}%
_hL_{lk}^i=0$ and its conjugated $\dot{\partial}_{\bar{h}}L_{\bar{l}\bar{k}%
}^{\bar{\imath}}=0.$ This means that $L_{\bar{l}\bar{k}}^{\bar{\imath}}$ are
holomorphic in $\eta $, which together with their $0$ - homogeneity gives $%
L_{jk}^i(z).$ Thus, by Lemma 2.1 ii) it results $C_{l\bar{r}h|\bar{k}}=0.$
Conversely, if $C_{l\bar{r}h|\bar{k}}=0$ then ii) from Lemma 2.1 become $(%
\dot{\partial}_{\bar{h}}L_{lk}^i)g_{i\bar{r}}+(\dot{\partial}_{\bar{h}%
}N_k^i)C_{i\bar{r}l}=0.$ It contracted by $\eta ^l$ gives $\dot{\partial}_{%
\bar{h}}N_k^i=0$ and so, $\dot{\partial}_{\bar{h}}L_{lk}^i=0$ which implies $%
L_{jk}^i(z).$ Now, using i) from Lemma 2.1, we obtain $C_{l\bar{r}%
h|k}=0.$
\end{proof}

We note that $C_{l\bar{r}h|\bar{k}}=0$ or $C_{l\bar{r}h|k}=0$ implies $\dot{%
\partial}_{\bar{h}}G^i=0$, but the converse is not true. $\dot{\partial}_{%
\bar{h}}G^i=0$ together with the K\"{a}hler condition gives $C_{l\bar{r}h|%
\bar{k}}=0$ or $C_{l\bar{r}h|k}=0.$ Therefore, some tensorial
characterizations for complex Berwald spaces are contained in the next
theorem.

\begin{theorem}
$(M,F)$ is a complex Berwald space if and only if it is K\"{a}hler and
either $C_{l\bar{r}h|\bar{k}}=0$ or $C_{l\bar{r}h|k}=0$.
\end{theorem}

In the remainder of this section we return to the notion of the real Berwald
space, \cite{Be}. It is a real Finsler space for which the coefficients of
the (real) Berwald connection depend only on the position. Our problem is to
see whether there exist a corespondent of this real assertion in complex
Finsler spaces. Taking into account Theorem 3.4 we have $\stackrel{B}{%
L_{jk}^i}(z)$, for any complex Berwald space. Nevertheless the converse is
not true, will see below there are complex Finsler spaces with $\stackrel{B}{%
L_{jk}^i}$ depending only on $z$ which are not Berwald. Therefore, it comes
into view another class of complex Finsler spaces.

\begin{definition}
Let $(M,F)$ be a $n$ - dimensional complex Finsler space. $(M,F)$ is called
generalized Berwald if the Berwald connection coefficients $\stackrel{B}{%
L_{jk}^{i}}$ depend only on the position $z.$
\end{definition}

Using Corollary 3.1 and Proposition 3.2, we have proved the following,

\begin{theorem}
Let $(M,F)$ be a $n$ - dimensional complex Finsler space. Then the following
assertions are equivalent:

i) $(M,F)$ is generalized Berwald;

ii) $G^{i}$ are holomorphic in $\eta ;$

iii) $B\Gamma $ is of $(1,0)$ - type.
\end{theorem}

\begin{corollary}
If $(M,F)$ is a complex Finsler space with the $C-F$ coefficients $%
L_{jk}^{i} $ depending only on $z,$ then the space is generalized Berwald.
\end{corollary}

\begin{proof}
We have $\dot{\partial}_{\bar{h}}L_{jk}^i=0$, which contracted by $%
\eta ^j\eta ^k$ gives $\dot{\partial}_{\bar{h}}G^i=0.$
\end{proof}

An example of generalized Berwald space is given by the complex version of
\textit{Antonelli - Shimada }metric
\begin{equation}
F_{AS}^2=L_{AS}(z,w;\eta ,\theta ):=e^{2\sigma }\left( |\eta |^4+|\theta
|^4\right) ^{\frac 12},\;\mbox{with}\;\;\eta ,\theta \neq 0,\;\;
\label{IV.9}
\end{equation}
on a domain $D$ from $\widetilde{T^{\prime }M},$ $\dim _CM=2,$ such that its
metric tensor is nondegenerated. We relabeled the local coordinates $z^1,$ $%
z^2,$ $\eta ^1,$ $\eta ^2$ as $z,$ $w,\eta ,$ $\theta ,$ respectively. $%
\sigma (z,w)$ is a real valued function and $|\eta ^i|^2:=\eta ^i\overline{%
\eta }^i,$ $\eta ^i\in \{\eta ,\theta \}$, (\cite{Mub}). A direct
computation leads to
\[
L_{zz}^z=L_{wz}^w=2\frac{\partial \sigma }{\partial z};\;L_{zw}^z=L_{ww}^w=2%
\frac{\partial \sigma }{\partial w},
\]
which depend only on $z$ and $w.$

\medskip Summing up all the results proved above we have the inclusions from
Fig. 1.

The intersection between the set of Landsberg spaces and those of the
generalized Berwald spaces gives the class of $G$ - Landsberg spaces.

Trivial examples of such spaces are given by the purely Hermitian and local
Minkowski manifolds. In the next section we came with some nice families of
generalized Berwald spaces.

\section{Generalized Berwald spaces with $(\alpha ,\beta )$ - metrics}

We consider $z\in M,$ $\eta \in T_z^{\prime }M,$ $\eta =\eta ^i\frac
\partial {\partial z^i}$, $\tilde{a}:=a_{i\bar{j}}(z)dz^i\otimes d\bar{z}^j$
a purely Hermitian metric and $b=b_i(z)dz^i$ a differential ($1,0)-$ form.
By these objects we have defined (for more details see \cite{Al-Mu, Al-Mu2})
the complex $(\alpha ,\beta )-$ metric $F$ on $T^{\prime }M$%
\begin{equation}
F(z,\eta ):=F(\alpha (z,\eta ),|\beta (z,\eta )|),  \label{3.1}
\end{equation}
where
\begin{equation}
\alpha (z,\eta ):=\sqrt{a_{i\bar{j}}(z)\eta ^i\bar{\eta}^j}\mbox{
and   }\beta (z,\eta )=b_i(z)\eta ^i.  \nonumber
\end{equation}

Let us recall that the coefficients of the $C-F$ connection corresponding to
the purely Hermitian metric $\alpha $ are
\[
\stackrel{a}{N_j^k}:=a^{\bar{m}k}\frac{\partial a_{l\bar{m}}}{\partial z^j}%
\eta ^l\;;\;\stackrel{a}{L_{jk}^i}:=a^{\overline{l}i}(\stackrel{a}{\delta }%
_ka_{j\overline{l}})\;;\;\stackrel{a}{C_{jk}^i}=0
\]
and we consider the settings
\begin{equation}
b^i:=a^{\bar{j}i}b_{\bar{j}}\;\;;\;\;||b||^2:=a^{\bar{j}i}b_ib_{\bar{j}%
}\;\;;\;\;b^{\bar{\imath}}:=\bar{b}^i.  \nonumber
\end{equation}

\begin{lemma}
\cite{Al-Mu} Let $(M,F)$ be a complex Finsler space with $(\alpha ,\beta )-$
metrics which satisfy $\displaystyle \frac{\partial |\beta |^{2}}{\partial
z^{i}}=||b||^{2}\frac{\partial \alpha ^{2}}{\partial z^{i}}.$ The following
statements are equivalent:

i) $\displaystyle ||b||^{2}b^{r}\frac{\partial a_{r\bar{m}}}{\partial z^{l}}%
\bar{\eta}^{m}=\bar{\beta}\bar{b}^{m}b^{r}\frac{\partial a_{r\bar{m}}}{%
\partial z^{l}};$

ii) $\displaystyle ||b||^{2}\frac{\partial b_{\bar{m}}}{\partial z^{i}}\bar{%
\eta}^{m}=\bar{\beta}\frac{\partial b_{\bar{m}}}{\partial z^{i}}\bar{b}^{m};$

iii) $\displaystyle b_{\bar{s}}\frac{\partial b_{\bar{m}}}{\partial z^{i}}%
\bar{\eta}^{m}=\bar{\beta}\frac{\partial b_{\bar{s}}}{\partial z^{i}};$

iv) $\displaystyle \bar{\beta}\left( \frac{\partial b_{i}}{\partial z^{l}}%
\eta ^{i}\eta ^{l}-2b_{l}\stackrel{a}{G^{l}}\right) +\beta \frac{\partial b_{%
\bar{m}}}{\partial z^{l}}\bar{\eta}^{m}\eta ^{l}=0$, where $\stackrel{a}{%
G^{l}}:=\frac{1}{2}\stackrel{a}{N_{j}^{l}}\eta ^{j}.$
\end{lemma}

\begin{proposition}
\cite{Al-Mu} Let $(M,F)$ be a complex Finsler space with $(\alpha ,\beta )-$
metrics. If $\frac{\partial |\beta |^{2}}{\partial z^{i}}=||b||^{2}\frac{%
\partial \alpha ^{2}}{\partial z^{i}}$ and one of the equivalent conditions
from Lemma 4.1 holds, then $N_{j}^{i}=\stackrel{a}{N_{j}^{i}}.$ Moreover, if
$\alpha $ is K\"{a}hler, then $F$ is K\"{a}hler.
\end{proposition}

\begin{theorem}
Let $(M,F)$ be a complex Finsler space with $(\alpha ,\beta )-$ metrics. If $%
\frac{\partial |\beta |^{2}}{\partial z^{i}}=||b||^{2}\frac{\partial \alpha
^{2}}{\partial z^{i}}$ and one of the equivalent conditions from Lemma 4.1
holds, then the space is generalized Berwald. Moreover if $\alpha $ is
K\"{a}hler then the space is Berwald.
\end{theorem}

\begin{proof}
By Proposition 4.1 we have $G^i=a^{\bar{m}k}\frac{\partial a_{l\bar{m}%
}}{\partial z^j}\eta ^l\eta ^j$ which is holomorphic in $\eta ,$ i.e. the
space is generalized Berwald. Adding the K\"{a}hler property for $a_{i%
\bar{j}}$ then the space becomes one complex Berwald.
\end{proof}

Further on, we asunder focus on two classes of complex $(\alpha ,\beta )-$
metrics, namely the complex Randers metrics $F:=\alpha +|\beta |$ and the
complex Kropina metrics $F:=\frac{\alpha ^2}{|\beta |},$ $|\beta |\neq 0.$

\subsection{Complex Randers metric $F:=\alpha +|\beta |$}

For the complex Randers metric $F:=\alpha +|\beta |$ we have, (\cite{Al-Mu2}%
)
\begin{eqnarray}
\frac{\partial \alpha }{\partial \eta ^i} &=&\frac 1{2\alpha }l_i\;\;;\;\;%
\frac{\partial |\beta |}{\partial \eta ^i}=\frac{\bar{\beta}}{2|\beta |}%
b_i\;;\;\;\eta _i:=\frac{\partial L}{\partial \eta ^i}=\frac F\alpha l_i+%
\frac{F\bar{\beta}}{|\beta |}b_i,  \label{4.1} \\
N_j^i &=&\stackrel{a}{N_j^i}+\frac 1\gamma \left( l_{\bar{r}}\frac{\partial
\bar{b}^r}{\partial z^j}-\frac{\beta ^2}{|\beta |^2}\frac{\partial b_{\bar{r}%
}}{\partial z^j}\bar{\eta}^r\right) \xi ^i+\frac \beta {2|\beta |}k^{%
\overline{r}i}\frac{\partial b_{\bar{r}}}{\partial z^j},  \nonumber
\end{eqnarray}
where $k^{\bar{r}i}:=2\alpha a^{\bar{j}i}+\frac{2(\alpha ||b||^2+2|\beta |)}%
\gamma \eta ^i\bar{\eta}^r-\frac{2\alpha ^3}\gamma b^i\bar{b}^r-\frac{%
2\alpha }\gamma (\bar{\beta}\eta ^i\bar{b}^r+\beta b^i\bar{\eta}^r)$, $%
\gamma :=L+\alpha ^2(||b||^2-1)$, $\xi ^i:=\bar{\beta}\eta ^i+\alpha ^2b^i$
and so, the spray coefficients are
\begin{equation}
G^i=\stackrel{a}{G^i}+\frac 1{2\gamma }\left( l_{\bar{r}}\frac{\partial \bar{%
b}^r}{\partial z^j}-\frac{\beta ^2}{|\beta |^2}\frac{\partial b_{\bar{r}}}{%
\partial z^j}\bar{\eta}^r\right) \xi ^i\eta ^j+\frac \beta {4|\beta |}k^{%
\overline{r}i}\frac{\partial b_{\bar{r}}}{\partial z^j}\eta ^j.  \label{4.2}
\end{equation}

Moreover, for the weakly K\"{a}hler complex Randers spaces we have proven.

\begin{proposition}
(\cite{Al-Mu2}) A complex Randers space $(M,F)$ is weakly K\"{a}hler if and
only if
\[
\frac{\alpha ^{2}|\beta |}{\gamma \delta }\left[ \beta \frac{\alpha
||b||^{2}+|\beta |}{|\beta |}\frac{\partial b_{\bar{m}}}{\partial z^{r}}\bar{%
\eta}^{m}+\bar{\beta}\left( \frac{\partial b_{r}}{\partial z^{l}}-b^{\bar{m}}%
\frac{\partial a_{l\bar{m}}}{\partial z^{r}}\right) \eta ^{l}-\alpha |\beta
|b^{\bar{m}}\frac{\partial b_{\bar{m}}}{\partial z^{r}}\right] \eta
^{r}C_{k}
\]
\begin{equation}
-\left( \alpha \bar{\beta}F_{kl}+\alpha b_{l}\frac{\partial b_{\bar{r}}}{%
\partial z^{k}}\bar{\eta}^{r}+2|\beta |a_{l\bar{r}}\Gamma _{\bar{j}k}^{\bar{r%
}}\bar{\eta}^{j}\right) \eta ^{l}+\alpha b_{k}\frac{\partial b_{\bar{m}}}{%
\partial z^{r}}\bar{\eta}^{m}\eta ^{r}=0,  \label{4.3}
\end{equation}
where $C_{j}:=C_{j\bar{h}k}g^{\bar{h}k}=\delta \left( \frac{1}{\alpha ^{2}}%
l_{j}-\frac{\bar{\beta}}{|\beta |^{2}}b_{j}\right) $ with $\delta :=\frac{%
\alpha ^{2}||b||^{2}-|\beta |^{2}}{2\gamma }-\frac{n|\beta |}{2F},$ $\Gamma
_{\bar{j}i}^{\bar{r}}:=\frac{1}{2}a^{\bar{r}k}(\frac{\partial a_{k\bar{j}}}{%
\partial z^{i}}$ $-\frac{\partial a_{i\bar{j}}}{\partial z^{k}})$ and $%
F_{il}:=\frac{\partial b_{l}}{\partial z^{i}}-\frac{\partial bi}{\partial
z^{l}}$.
\end{proposition}

\begin{theorem}
Let $(M,F)$ be a connected complex Randers space. Then, $(M,F)$ is a
generalized Berwald space if and only if $(\bar{\beta}l_{\bar{r}}\frac{%
\partial \bar{b}^{r}}{\partial z^{j}}+\beta \frac{\partial b_{\bar{r}}}{%
\partial z^{j}}\bar{\eta}^{r})\eta ^{j}=0.$
\end{theorem}

\begin{proof}
If $(M,F)$ is generalized Berwald then $2G^i=\stackrel{B}{L_{jk}^i}(z)
$ $\eta ^j\eta ^k,$ which means that $G^i$ is quadratic in $\eta .$ Thus,
using (\ref{4.2}) we have

$\alpha |\beta |\mathbf{\{}-\beta \mathbf{[}(\alpha ^2||b||^2+|\beta |^2)a^{%
\bar{r}i}+||b||^2\bar{\eta}^r\eta ^i-\alpha ^2\bar{b}^rb^i-\bar{\beta}\eta ^i%
\bar{b}^r-\beta b^i\bar{\eta}^r\mathbf{]}\frac{\partial b_{\bar{r}}}{%
\partial z^j}\eta ^j$

$+4|\beta |^2(G^i-\stackrel{a}{G^i})\mathbf{\}}+|\beta |^2\mathbf{[}2(\alpha
^2||b||^2+|\beta |^2)(G^i-\stackrel{a}{G^i})-2\alpha ^2\beta a^{\bar{r}i}%
\frac{\partial b_{\bar{r}}}{\partial z^j}\eta ^j$

$-(\bar{\beta}l_{\bar{r}}\frac{\partial \bar{b}^r}{\partial z^j}+\beta \frac{%
\partial b_{\bar{r}}}{\partial z^j}\bar{\eta}^r)\eta ^j\eta ^i-\frac{\alpha
^2\beta }{|\beta |^2}(\bar{\beta}l_{\bar{r}}\frac{\partial \bar{b}^r}{%
\partial z^j}-\beta \frac{\partial b_{\bar{r}}}{\partial z^j}\bar{\eta}%
^r)\eta ^jb^i\mathbf{]}=0,$\newline
which contains an irrational part and another rational one. Thus, we obtain
\begin{eqnarray*}
&&\beta \left[ (\alpha ^2||b||^2+|\beta |^2)a^{\bar{r}i}+||b||^2\bar{\eta}%
^r\eta ^i-\alpha ^2\bar{b}^rb^i-\bar{\beta}\eta ^i\bar{b}^r-\beta b^i\bar{%
\eta}^r\right] \frac{\partial b_{\bar{r}}}{\partial z^j}\eta ^j \\
&=&4|\beta |^2(G^i-\stackrel{a}{G^i})\;\mbox{  and  }\; \\
&&(\bar{\beta}l_{\bar{r}}\frac{\partial \bar{b}^r}{\partial z^j}+\beta \frac{%
\partial b_{\bar{r}}}{\partial z^j}\bar{\eta}^r)\eta ^j\eta ^i+\frac{\alpha
^2\beta }{|\beta |^2}(\bar{\beta}l_{\bar{r}}\frac{\partial \bar{b}^r}{%
\partial z^j}-\beta \frac{\partial b_{\bar{r}}}{\partial z^j}\bar{\eta}%
^r)\eta ^jb^i+2\alpha ^2\beta a^{\bar{r}i}\frac{\partial b_{\bar{r}}}{%
\partial z^j}\eta ^j \\
&=&2(\alpha ^2||b||^2+|\beta |^2)(G^i-\stackrel{a}{G^i}).
\end{eqnarray*}
Theirs contractions by $b_i$ and $l_i$ yield
\begin{equation}
(G^i-\stackrel{a}{G^i})b_i=0;  \label{4.3'}
\end{equation}
\begin{eqnarray}
4|\beta |^2(\stackrel{a}{G^i}-G^i)l_i+2\beta \alpha ^2(||b||^2\bar{\eta}^r-%
\bar{\beta}\bar{b}^r)\frac{\partial b_{\bar{r}}}{\partial z^j}\eta ^j &=&0;
\nonumber \\
\bar{\beta}(\alpha ^2||b||^2+|\beta |^2)l_{\bar{r}}\frac{\partial \bar{b}^r}{%
\partial z^j}\eta ^j-\beta (\alpha ^2||b||^2-|\beta |^2)\frac{\partial b_{%
\bar{r}}}{\partial z^j}\bar{\eta}^r\eta ^j+2\alpha ^2|\beta |^2\bar{b}^r%
\frac{\partial b_{\bar{r}}}{\partial z^j}\eta ^j &=&0;  \nonumber \\
(\alpha ^2||b||^2+|\beta |^2)(\stackrel{a}{G^i}-G^i)l_i+\alpha ^2(\bar{\beta}%
l_{\bar{r}}\frac{\partial \bar{b}^r}{\partial z^j}+\beta \frac{\partial b_{%
\bar{r}}}{\partial z^j}\bar{\eta}^r)\eta ^j &=&0.  \nonumber
\end{eqnarray}

Adding the second and the third relations from (\ref{4.3'}), we obtain

$4|\beta |^2(\stackrel{a}{G^i}-G^i)l_i+(\alpha ^2||b||^2+|\beta |^2)(\bar{%
\beta}l_{\bar{r}}\frac{\partial \bar{b}^r}{\partial z^j}+\beta \frac{%
\partial b_{\bar{r}}}{\partial z^j}\bar{\eta}^r)\eta ^j=0.$

This together with the fourth equation from (\ref{4.3'}) implies $(\stackrel{%
a}{G^i}-G^i)l_i=0$ and $(\bar{\beta}l_{\bar{r}}\frac{\partial \bar{b}^r}{%
\partial z^j}+\beta \frac{\partial b_{\bar{r}}}{\partial z^j}\bar{\eta}%
^r)\eta ^j=0.$

Conversely, if $(\bar{\beta}l_{\bar{r}}\frac{\partial \bar{b}^r}{\partial z^j%
}+\beta \frac{\partial b_{\bar{r}}}{\partial z^j}\bar{\eta}^r)\eta ^j=0,$ by
derivation with respect to $\bar{\eta}^m$ we deduce $(l_{\bar{r}}\frac{%
\partial \bar{b}^r}{\partial z^j}b_{\bar{m}}+\beta \frac{\partial b_{\bar{m}}%
}{\partial z^j})\eta ^j=0.$ The last two relations give
\[
a^{\bar{m}i}\frac{\partial b_{\bar{m}}}{\partial z^j}\eta ^j=\frac \beta
{|\beta |^2}\frac{\partial b_{\bar{r}}}{\partial z^j}\bar{\eta}^rb^i\eta ^j\;%
\mbox{  and  }\;\bar{b}^m\frac{\partial b_{\bar{m}}}{\partial z^j}\eta
^j=||b||^2\frac \beta {|\beta |^2}\frac{\partial b_{\bar{r}}}{\partial z^j}%
\bar{\eta}^r\eta ^j,
\]
which substituted into (\ref{4.2}) imply $G^i=\stackrel{a}{G^i}$ and so, $G^i
$ are holomorphic in $\eta ,$ i.e. the space is generalized Berwald.
\end{proof}

\begin{theorem}
Let $(M,F)$ be a connected complex Randers space. Then, $(M,F)$ is a complex
Berwald space if and only if it is in the same time generalized Berwald and
weakly K\"{a}hler.
\end{theorem}

\begin{proof}
If $(M,F)$ is Berwald then it is obvious that the space is
generalized Berwald and weakly K\"{a}hler.

Now, we prove the converse. On the one hand, if the space is generalized
Berwald, by Theorem 4.2, it results $(\bar{\beta}l_{\bar{r}}\frac{\partial
\bar{b}^r}{\partial z^j}+\beta \frac{\partial b_{\bar{r}}}{\partial z^j}\bar{%
\eta}^r)\eta ^j=0,$ which can be rewritten as
\begin{equation}
\bar{\beta}\left( \frac{\partial b_i}{\partial z^l}\eta ^i\eta ^l-2b_l%
\stackrel{a}{G^l}\right) +\beta \frac{\partial b_{\bar{m}}}{\partial z^l}%
\bar{\eta}^m\eta ^l=0.  \label{4.4}
\end{equation}
Moreover, (\ref{4.4}) implies

\begin{equation}
||b||^2\bar{\beta}\left( \frac{\partial b_i}{\partial z^l}\eta ^i\eta ^l-2b_l%
\stackrel{a}{G^l}\right) +|\beta |^2\bar{b}^m\frac{\partial b_{\bar{m}}}{%
\partial z^l}\eta ^l=0.  \label{4.5}
\end{equation}

On the second hand, the space is supposed weakly K\"{a}hler. Therefore, (\ref
{4.4}) and (\ref{4.5}) substituted into (\ref{4.3}) lead to
\begin{equation}
\alpha ^2\left( \bar{\beta}F_{kl}\eta ^l+\beta \frac{\partial b_{\bar{r}}}{%
\partial z^k}\bar{\eta}^r-b_k\frac{\partial b_{\bar{r}}}{\partial z^l}\bar{%
\eta}^r\eta ^l\right) +2\alpha |\beta |a_{l\bar{r}}\Gamma _{\bar{j}k}^{\bar{r%
}}\bar{\eta}^j=0,  \label{4.6}
\end{equation}
which contains two parts: the first is rational and the second is
irrational. It results
\begin{equation}
\bar{\beta}F_{kl}\eta ^l+\beta \frac{\partial b_{\bar{r}}}{\partial z^k}\bar{%
\eta}^r-b_k\frac{\partial b_{\bar{m}}}{\partial z^l}\bar{\eta}^m\eta ^l=0%
\mbox{  and  }a_{l\bar{r}}\Gamma _{\bar{j}k}^{\bar{r}}\bar{\eta}^j=0.
\label{4.7}
\end{equation}
The second condition from (\ref{4.7}) gives the K\"{a}hler property for $%
\alpha .$ Thus, deriving (\ref{4.4}) with respect to $\eta ^k$ it follows
\begin{equation}
b_k\frac{\partial b_{\bar{r}}}{\partial z^l}\bar{\eta}^r\eta ^l=-\beta \frac{%
\partial b_{\bar{r}}}{\partial z^k}\bar{\eta}^r-\bar{\beta}\left( \frac{%
\partial b_l}{\partial z^k}+\frac{\partial b_k}{\partial z^l}\right) \eta
^l+2\bar{\beta}b_l\stackrel{a}{N_k^l.}  \label{4.8}
\end{equation}
Now, (\ref{4.8}) together with the first condition from (\ref{4.7}) implies
\begin{equation}
\bar{\beta}l_{\bar{r}}\frac{\partial \bar{b}^r}{\partial z^k}+\beta \frac{%
\partial b_{\bar{r}}}{\partial z^k}\bar{\eta}^r=0  \label{4.9}
\end{equation}
and from here results its derivative with respect to $\bar{\eta}^m$%
\begin{equation}
l_{\bar{r}}\frac{\partial \bar{b}^r}{\partial z^k}b_{\bar{m}}+\beta \frac{%
\partial b_{\bar{m}}}{\partial z^k}=0.  \label{4.10}
\end{equation}
Moreover, (\ref{4.9}) and (\ref{4.10}) imply
\begin{equation}
a^{\bar{m}i}\frac{\partial b_{\bar{m}}}{\partial z^k}=\frac \beta {|\beta
|^2}\frac{\partial b_{\bar{r}}}{\partial z^k}\bar{\eta}^rb^i\;\mbox{  and  }%
\;\bar{b}^m\frac{\partial b_{\bar{m}}}{\partial z^k}=||b||^2\frac \beta
{|\beta |^2}\frac{\partial b_{\bar{r}}}{\partial z^k}\bar{\eta}^r.
\label{4.11}
\end{equation}
Plugging (\ref{4.9}) and (\ref{4.11}) into (\ref{4.1}), we obtain $N_j^i=%
\stackrel{a}{N_j^i}$ and so, $L_{kj}^i=\stackrel{a}{L_{kj}^i}=$ $\stackrel{a%
}{L_{jk}^i}=L_{jk}^i,$ i.e. the Randers space is K\"{a}hler which proves our claim.
\end{proof}

\subsection{Complex Kropina metric $F:=\frac{\alpha ^2}{|\beta |},|\beta
|\neq 0$}

For the complex Kropina metric $F:=\frac{\alpha ^2}{|\beta |},|\beta |\neq
0, $ we have, (\cite{Al2})
\begin{eqnarray}
\frac{\partial \alpha }{\partial \eta ^i} &=&\frac 1{2\alpha }l_i;\;\frac{%
\partial |\beta |}{\partial \eta ^i}=\frac{\bar{\beta}}{2|\beta |}b_i;\;\eta
_i:=\frac{\partial L}{\partial \eta ^i}=2q^2l_i-q^4\bar{\beta}b_i;\;q:=\frac
\alpha {|\beta |};  \label{4.12} \\
N_j^i &=&\stackrel{a}{N_j^i}-\frac{\bar{\beta}}{|\beta |^2}l_{\bar{r}}\frac{%
\partial \bar{b}^r}{\partial z^j}\eta ^i-\frac{q^2\beta }2t^{\overline{r}i}%
\frac{\partial b_{\bar{r}}}{\partial z^j},  \nonumber
\end{eqnarray}
where $t^{\overline{r}i}:=a^{\overline{r}i}+\frac{2-q^2||b||^2}{q^2|\beta |^2%
}\eta ^i\bar{\eta}^r+\frac 1{|\beta |^2}(\bar{\beta}\eta ^i\bar{b}^r-\beta
b^i\bar{\eta}^r)$ and so, the spray coefficients are
\begin{equation}
G^i=\stackrel{a}{G^i}-\frac{\bar{\beta}}{2|\beta |^2}l_{\bar{r}}\frac{%
\partial \bar{b}^r}{\partial z^j}\eta ^i\eta ^j-\frac{q^2\beta }4t^{%
\overline{r}i}\frac{\partial b_{\bar{r}}}{\partial z^j}\eta ^j.  \label{4.13}
\end{equation}

\begin{proposition}
Let $(M,F)$ be a connected complex Kropina space. Then, $G^{i}=\stackrel{a}{%
G^{i}}$ if and only if $(\bar{\beta}l_{\bar{r}}\frac{\partial b^{\bar{r}}}{%
\partial z^{j}}+\beta \frac{\partial b_{\bar{r}}}{\partial z^{j}}\bar{\eta}%
^{r})\eta ^{j}=0.$
\end{proposition}

\begin{proof}
Using (\ref{4.13}), we have

$4|\beta |^2(\stackrel{a}{G^i}-G^i)l_i=2\alpha ^2(\bar{\beta}l_{\bar{r}}%
\frac{\partial b^{\bar{r}}}{\partial z^j}+\beta \frac{\partial b_{\bar{r}}}{%
\partial z^j}\bar{\eta}^r)\eta ^j+\alpha ^4\frac{\partial b_{\bar{r}}}{%
\partial z^j}\eta ^j(b^{\bar{r}}-\frac \beta {|\beta |^2}||b||^2\bar{\eta}^r)
$ and

$4|\beta |^2(\stackrel{a}{G^i}-G^i)b_i=2\beta (\bar{\beta}l_{\bar{r}}\frac{%
\partial b^{\bar{r}}}{\partial z^j}+\beta \frac{\partial b_{\bar{r}}}{%
\partial z^j}\bar{\eta}^r)\eta ^j+2\alpha ^2\beta \frac{\partial b_{\bar{r}}%
}{\partial z^j}\eta ^j(b^{\bar{r}}-\frac \beta {|\beta |^2}||b||^2\bar{\eta}%
^r).$

Thus, if $G^i=\stackrel{a}{G^i}$ then $(\bar{\beta}l_{\bar{r}}\frac{\partial
b^{\bar{r}}}{\partial z^j}+\beta \frac{\partial b_{\bar{r}}}{\partial z^j}%
\bar{\eta}^r)\eta ^j.$ Conversely, the assumption $(\bar{\beta}l_{\bar{r}}%
\frac{\partial b^{\bar{r}}}{\partial z^j}+\beta \frac{\partial b_{\bar{r}}}{%
\partial z^j}\bar{\eta}^r)\eta ^j=0$ implies $a^{\bar{m}i}\frac{\partial b_{%
\bar{m}}}{\partial z^j}\eta ^j=\frac \beta {|\beta |^2}\frac{\partial b_{%
\bar{r}}}{\partial z^j}\bar{\eta}^rb^i\eta ^j\;$and$\;\bar{b}^m\frac{%
\partial b_{\bar{m}}}{\partial z^j}\eta ^j=||b||^2\frac \beta {|\beta |^2}%
\frac{\partial b_{\bar{r}}}{\partial z^j}\bar{\eta}^r\eta ^j,$ which
substituted into (\ref{4.13}) it results $G^i=\stackrel{a}{G^i}$.
\end{proof}

\begin{theorem}
If $(M,F)$ is a connected complex Kropina space with property $(\bar{\beta}%
l_{\bar{r}}\frac{\partial b^{\bar{r}}}{\partial z^{j}}+\beta \frac{\partial
b_{\bar{r}}}{\partial z^{j}}\bar{\eta}^{r})\eta ^{j}=0,$ then it is
generalized Berwald.
\end{theorem}

\begin{proof}
Indeed, by Proposition 4.3 we have $G^i=\stackrel{a}{G^i}$ which gives $%
\dot{\partial}_{\bar{h}}G^i=0$.
\end{proof}

\begin{proposition}
Let $(M,F)$ be a connected complex Kropina space. Then, $N_{j}^{i}=\stackrel{%
a}{N_{j}^{i}}$ if and only if $\bar{\beta}l_{\bar{r}}\frac{\partial b^{\bar{r%
}}}{\partial z^{j}}+\beta \frac{\partial b_{\bar{r}}}{\partial z^{j}}\bar{%
\eta}^{r}=0.$
\end{proposition}

\begin{proof}
Taking into account (\ref{4.12}), we obtain

$2|\beta |^2(\stackrel{a}{N_j^i}-N_j^i)l_i=2\alpha ^2(\bar{\beta}l_{\bar{r}}%
\frac{\partial b^{\bar{r}}}{\partial z^j}+\beta \frac{\partial b_{\bar{r}}}{%
\partial z^j}\bar{\eta}^r)+\alpha ^4\frac{\partial b_{\bar{r}}}{\partial z^j}%
(b^{\bar{r}}-\frac \beta {|\beta |^2}||b||^2\bar{\eta}^r)$ and

$2|\beta |^2(\stackrel{a}{N_j^i}-N_j^i)b_i=2\beta (\bar{\beta}l_{\bar{r}}%
\frac{\partial b^{\bar{r}}}{\partial z^j}+\beta \frac{\partial b_{\bar{r}}}{%
\partial z^j}\bar{\eta}^r)+2\alpha ^2\beta \frac{\partial b_{\bar{r}}}{%
\partial z^j}(b^{\bar{r}}-\frac \beta {|\beta |^2}||b||^2\bar{\eta}^r),$
which give $\bar{\beta}l_{\bar{r}}\frac{\partial b^{\bar{r}}}{\partial z^j}%
+\beta \frac{\partial b_{\bar{r}}}{\partial z^j}\bar{\eta}^r=0,$ under
assumption $N_j^i=\stackrel{a}{N_j^i}.$

Conversely, if $\bar{\beta}l_{\bar{r}}\frac{\partial b^{\bar{r}}}{\partial
z^j}+\beta \frac{\partial b_{\bar{r}}}{\partial z^j}\bar{\eta}^r=0$ then $a^{%
\bar{m}i}\frac{\partial b_{\bar{m}}}{\partial z^k}=\frac \beta {|\beta |^2}%
\frac{\partial b_{\bar{r}}}{\partial z^k}\bar{\eta}^rb^i$ and $\bar{b}^m%
\frac{\partial b_{\bar{m}}}{\partial z^k}=||b||^2\frac \beta {|\beta |^2}%
\frac{\partial b_{\bar{r}}}{\partial z^k}\bar{\eta}^r$ , which together (\ref
{4.12}) lead to $N_j^i=\stackrel{a}{N_j^i}$.
\end{proof}

\begin{theorem}
Let $(M,F)$ be a connected complex Kropina space. If $\alpha $ is K\"{a}hler
and $\bar{\beta}l_{\bar{r}}\frac{\partial b^{\bar{r}}}{\partial z^{j}}+\beta
\frac{\partial b_{\bar{r}}}{\partial z^{j}}\bar{\eta}^{r}=0$, then the space
is Berwald.
\end{theorem}

\begin{proof}
By Proposition 4.4 we have $N_j^i=\stackrel{a}{N_j^i}$ and so, $\dot{%
\partial}_{\bar{h}}G^i=0,$ i.e. the space is generalized Berwald. But $%
\alpha $ is supposed K\"{a}hler, therefore it results that $F$ is also K\"{a}%
hler. The proof is complete.
\end{proof}

\medskip Certainly, as it is expected, our study is far from being complete.
Here we tried to point out some classes of complex Finsler spaces with
special properties for Cartan tensors, having in mind an analogy with the
real case. There is not enough space here to prove with other examples that
our classification of these complex Finsler spaces is proper. Otherwise, we
don't have at hand too many examples of complex Finsler spaces, in fact the
study of these spaces could be considered rather at a first stage. It is our
goal to look for other significant examples from the new class of complex
Randers spaces, \cite{Al-Mu2,C-S2}, in particular for two dimensional case,
recently studied by us. Keeping in mind that in the real case the relations
between Landsberg and Berwald spaces give rise to some questions which had
been open for a long time (\cite{Do, Sz2, Mt}), it is possible that the same
takes place for our setting. However, from our point of view this
classification seems natural.

\textbf{Acknowledgment:} This paper is supported by the Sectorial
Operational Program Human Resources Development (SOP HRD), financed from the
European Social Fund and by Romanian Government under the Project number
POSDRU/89/1.5/S/59323.

\begin{flushleft}

Transilvania Univ.,
Faculty of Mathematics and Informatics

Iuliu Maniu 50, Bra\c{s}ov 500091, ROMANIA

e-mail: nicoleta.aldea@lycos.com

e-mail: gh.munteanu@unitbv.ro
\end{flushleft}

\end{document}